\documentclass{svproc}

\usepackage{amsmath}
\usepackage{amssymb}
\usepackage[dvipsnames]{xcolor}
\usepackage{hyperref}
\usepackage{autonum}
\usepackage{bigints}
\usepackage{cite}
\usepackage{pgfplots}
\pgfplotsset{compat=1.18}
\usepackage{pgfplotstable}
\usepackage{subcaption}
\usepackage{algorithm2e}
\usepackage{dirtytalk}

\allowdisplaybreaks

\usepackage{url}

\newcommand{\transp}{\mathsf{T}\!}
\begin{document}

\mainmatter              


\title{An Investigation into the Distribution of Ratios of Particle Solver-based Likelihoods}
\titlerunning{An Investigation into Ratios of Particle Solver-based Likelihoods}  
%
\author{Emil~Løvbak \and Sebastian~Krumscheid}
\authorrunning{Emil Løvbak and Sebastian Krumscheid} 
\institute{Karlsruhe Institute of Technology, Zirkel 2, 76133 Karlsruhe, Germany\\
\email{\{emil.loevbak,sebastian.krumscheid\}@kit.edu}}

\maketitle              

\begin{abstract}
We investigate the use of the Metropolis-Hastings algorithm to sample posterior distribution in a Bayesian inverse problem, where the likelihood function is random. Concretely, we consider the case where one has full field observations of a PDE solution, in case a one-dimensional diffusion equation, subject to a Gaussian observation error. Assuming one uses a particle-based Monte Carlo simulation when approximating the likelihood function, one gets an approximate likelihood with additive Gaussian noise in the log-likelihood. We study how these two Gaussian distributions affect the distribution of ratios of approximate likelihood evaluations, as required when evaluating acceptance probabilities in the Metropolis-Hastings algorithm. We do so through both theoretical analysis and numerical experiments.

\end{abstract}

\section{Introduction}
\label{sec:introduction}

We consider a diffusion equation with homogeneous diffusion coefficient on the one-dimensional spatial domain $x \in [0,L)$ and periodic boundaries, i.e.,
\begin{gather}
	\label{eq:diff_eq}
	\partial_t \rho(x,t) = D \partial_{xx} \rho(x,t),  \\
	\rho(0, t) = \rho(L,t), \qquad \partial_x \rho(0, t) = \partial_x \rho(L, t), \qquad \rho(x,0) = \rho_0(x) \geq 0. \label{eq:diff_eq_boundary}
\end{gather}
Here, $t \in \mathbb{R}_{\geq 0}$ is the time variable and $\rho_0(x)$ is a given initial condition. We approximate the solution to this equation using a Monte Carlo simulation. Such simulations discretize the equation at hand using an ensemble of particles. The particle dynamics are chosen such that the ensemble statistics converge to a (potentially biased) approximation of the true solution of \eqref{eq:diff_eq}--\eqref{eq:diff_eq_boundary} as the particle ensemble size $P\to\infty$. Although Monte Carlo simulation can not compete with deterministic approaches, such as finite differences~\cite{LeVeque2007} or spectral methods~\cite{Kopriva2009} for the problem at hand, we are motivated by their indispensable character in various high-impact domains, including nuclear fusion research~\cite{Feng1997, Reiter2005}, financial mathematics~\cite{Glasserman2003} and computer graphics~\cite{Pharr2023}, considering the diffusion equation as a straightforward, low-cost example.

We aim to solve a discretized statistical inverse problem of the form
\begin{equation}
	\label{eq:inverseProblem}
	\rho = G[D] + \eta, \qquad \eta \sim \mathcal{N}(0,\Sigma_\eta),
\end{equation}
where, $G[\cdot]: \mathbb{R} \mapsto \mathbb{R}^N$ is the solution map of \eqref{eq:diff_eq}--\eqref{eq:diff_eq_boundary} for a given diffusion coefficient $D \in \mathbb{R}$ on a discrete spatial grid. The observation error $\eta \in \mathbb{R}^N$ is an unknown multivariate random quantity, assumed normally distributed with a strictly positive definite covariance matrix $\Sigma_\eta = \Sigma_\eta^\transp \in \mathbb{R}_{\succ 0}^{N \times N}$. Although $\eta$ is modeled as a random variable, it takes a fixed, but unknown, value for a given observation $\rho$.

We wish to determine a reasonable value of $D$, given an observation $\rho \in \mathbb{R}^N$. However, we assume that we only have an approximate solution map
\begin{equation}
	\label{eq:randomForwardMap}
	\hat{G}[D] = G[D] + \delta, \qquad \delta \sim \mathcal{N}(\mu(D), \Sigma(D)),
	\end{equation}
resulting from the use of Monte Carlo simulation. Here, $\delta \in \mathbb{R}^N$ follows a normal distribution with mean $\mu: \mathbb{R} \mapsto \mathbb{R}^N$ and covariance matrix $\Sigma = \Sigma^\transp: \mathbb{R} \mapsto \mathbb{R}^{N \times N}_{\succ 0}$, parameterized by the parameter $D$. In other words, $\delta$ is modeled as a white-noise Gaussian random process~\cite{Rasmussen2005} with respect to $D$. We motivate the assumed distribution in Section~\ref{sec:mcfordiffusion}.

As $\hat{G}[\cdot]$ is not easily invertible, solving the inverse problem \eqref{eq:inverseProblem} requires multiple evaluations of \eqref{eq:diff_eq}--\eqref{eq:diff_eq_boundary}~\cite{Tarantola2005}. This fact, combined with the high computational cost of Monte Carlo simulation means a cost-accuracy trade-off must be made in the value of $P$. This trade-off is further exacerbated when multiple evaluations of the model are needed, such as when solving problems of the form \eqref{eq:inverseProblem}. One must then take into account both the stochastic nature of the approximation error, as well as the generally high simulation cost incurred.

Due to this computational cost, most efforts to solve inverse problems using Monte Carlo simulation have relied on neglecting the observation error $\eta$, to solve the inverse problem as an optimization problem, minimizing the residual $G[D]-\rho$ under a suitable norm. The main challenge in the setting given by \eqref{eq:randomForwardMap}, where the forward model contains a solver with stochastic errors, is computing gradients. Various approaches have been applied to tackle this challenge, notably adjoint methods, see e.g.~\cite{Dekeyser2018, Loevbak2024, Li2023c, Caflisch2021, Caflisch2024}, and algorithmic differentiation, see e.g.~\cite{Horsten2024a, Yilmazer2024, Zhang2020a}. However, in this work, we consider a Bayesian approach, where we approximate a posterior distribution of $D$ conditioned on the observation $\rho$. The Bayesian posterior is viewed as the solution to the inverse problem as it encodes all of the available information about the unknown $D$ given the data $\rho$.

Using Bayes' rule, we find that the probability density function (PDF) of the posterior for $D$ and a given $\rho$ is
\begin{equation}
	\label{eq:bayes}
	\pi_\text{post}(D|\rho) = \frac{\pi_0(D)\pi_\text{l}(\rho|D)}{\pi_\text{obs}(\rho)},
\end{equation}
where $\pi_0(\cdot)$ is the PDF of an assumed prior distribution on $D$, $\pi_\text{l}(\rho|D)$ is the likelihood of $\rho$ given $D$, and $\pi_\text{obs}(\cdot)$ is the marginal PDF of the given data $\rho$. We note that $\pi_\text{obs}(\rho)$ is generally unknown.

In our setting, the posterior given by \eqref{eq:bayes} will depend nonlinearly on the random error $\delta$, due to using the approximate solution map \eqref{eq:randomForwardMap} in the log-likelihood. Given a deterministic discretization error $\delta$, one can apply standard textbook results to bound the resulting error on the posterior \eqref{eq:bayes} in, e.g., Hellinger distance~\cite[Ch. 1]{Sanz-Alonso2023}. Similar results on posterior convergence were proven in~\cite{Lie2018a} for random errors $\delta$, under suitable conditions, for the limit $\delta \to 0$. However, we consider the case of non-negligible stochastic errors. To the best of our knowledge, no theoretical results on well-posedness exist in this setting. We focus on algorithmic analysis in this work, rather than the posterior itself.

Markov chain Monte Carlo methods (MCMC) are designed to sample the posterior given by \eqref{eq:bayes}, despite the unknown denominator. Given $\rho$, these methods construct a Markov chain of correlated samples of $D$ where each transition in the chain evaluates ratios of posterior probabilities, hence making the denominators cancel when deciding to accept or reject a proposed sample. The most straightforward approach is the Metropolis-Hastings algorithm~\cite{Metropolis1953,Hastings1970}; however, many notable variants and extensions have been developed, e.g., Hamiltonian Monte Carlo~\cite{Duane1987} that incorporates gradient information.

In the considered setting with stochastic errors \eqref{eq:randomForwardMap}, we note that the use of MCMC for sampling the posterior \eqref{eq:bayes} results in a nested sampling problem as each MCMC sample requires a full Monte Carlo simulation. Our interest lies in understanding the effect of non-negligible noise in the forward map. For inherently random but unbiased likelihoods, one can rely on pseudo-marginal MCMC methods~\cite{Andrieu2009}. Such likelihoods are often the result of a product or sum of individual contributions from different data points~\cite{Maclaurin2015} or simulation outputs~\cite{Warne2020}. Pseudo-marginal methods consider the combined probability space of the parameter to be inferred and the stochastic dimension of the forward map. The aim is to approximate the marginal distribution of the inferred parameter $D$, i.e., integrating out the forward map's stochasticity, using samples from this combined probability space. Although these methods are well studied, the assumption of unbiasedness means that they cannot be directly applied in our setting. We also note that similar work on nested sampling has been conducted in the setting of optimal experimental design, where nested Monte Carlo sampling is used to marginalize over so-called nuisance parameters~\cite{Bartuska2022, Feng2019}.

The goal of this work is to consider the non-asymptotic sampling regime in $P$, i.e., simulations with non-negligible and possibly biased stochastic error. We study how the Monte Carlo simulation error affects the ratio of computed approximate likelihood values, as a first step towards understanding the acceptance-rejection behavior of the Metropolis-Hastings algorithm in this setting. Notably, we observe that this ratio becomes a random variable, whose distribution we study in terms of its moments. We derive an explicit expression for these moments when they exist, which can inform practitioners in selecting a suitable value for $P$ that balances computational cost with accuracy on approximated acceptance rates. This line of work complements that presented in~\cite{Martinek2024Manual}, applying MCMC to Monte Carlo simulations in photoacoustic imaging.

The remainder of this paper is structured as follows. In Section~\ref{sec:mcfordiffusion} we introduce a simple Monte Carlo discretization for the one-dimensional diffusion equation and make observations on the structure of the resulting discretization error. Next, in Section~\ref{sec:MCMC} we present a theoretical analysis of the ratio of approximate likelihoods evaluated using such solvers, in terms of the moments of the ratio's distribution. This analysis is then further supported by experimental results in Section~\ref{sec:numerics}. Finally, in Section~\ref{sec:conclusion} we draw some conclusions and discuss some their potential for producing future algorithmic developments.

\section{Monte Carlo simulation of the diffusion equation}
\label{sec:mcfordiffusion}

In this Section, we derive a Monte Carlo discretization of the model \eqref{eq:diff_eq}--\eqref{eq:diff_eq_boundary} and comment on the structure of the resulting error. This derivation is purposefully brief. For a more detailed mathematical derivation of Monte Carlo simulation for parabolic PDEs, we refer to textbooks, such as \cite{Graham2013}. We assume, for simplicity, that $\rho(x,t) \geq 0$. As such, we can interpret $\rho(x,t)$ as a probability density over $x$ at any given time $t$. We then make use of the fact that \eqref{eq:diff_eq} is the Fokker-Planck equation to the stochastic differential equation (SDE)
\begin{equation}
	\label{eq:SDE}
	\text{d} X(t) = \sqrt{2D} \text{d} W(t), \qquad X(0) \overset{\text{i.i.d.}}{\sim} \rho_0,
\end{equation}
with $W(t)$ a standard 1D Wiener process, and the initial condition $X(0)$ is distributed according to the density $\rho_0$.

Given \eqref{eq:SDE}, we construct an ensemble of $P$ particles, defined by their positions
\begin{equation}
	\label{eq:particleensemble}
	\{X_{p,k}\}_{p=1}^P
\end{equation} at time $t_k=k\Delta t$. Using an Euler-Maruyama discretization, we propagate each particle $p$ in \eqref{eq:particleensemble} as a discretized realization of \eqref{eq:SDE}, i.e., for $p=1,\dots,P$
\begin{equation}
	\label{eq:EulerMaruyama}
	X_{p,k+1} = X_{p,k} + \sqrt{2D\Delta t}W_{p,k}, \qquad W_{p,k} \overset{\text{i.i.d.}}{\sim} \mathcal{N}(0,1), \qquad k \geq 0,
\end{equation}
and	$X_{p,0} \overset{\text{i.i.d.}}{\sim} \rho_0$.
Simulating such an ensemble, produces an ensemble of particles, where the values $X_{p,k}$ are distributed according to $\rho(x,t)$ for any $t_k = k\Delta t$.

Once we have the ensemble of particles \eqref{eq:particleensemble}, we apply a binning strategy to compute a solution on a grid with $N = \frac{L}{\Delta x}$ cells We estimate the density at the center of the $n$-th grid cell $x_n = \left( n + \frac{1}{2} \right)\Delta x$ as
\begin{equation}
	\label{eq:binning}
	\rho(x_n,k\Delta t) \approx \hat{\rho}_{n,k} = \frac{1}{P\Delta x} \sum_{p=1}^P \mathcal{I}_n(X_{p,k}), \qquad n = 0,\dots,N-1,
\end{equation}
where $\mathcal{I}_n$ is the indicator function for the grid cell with index $n$. Binning straightforwardly produces a vector of length $N$ representing the solution at a given moment in time as a histogram, but introduces two sources of error: (i) the approximation of the solution by a piecewise constant function introduces a bias due to spatial discretization; (ii) each particle can only contribute to a single cell at each time step, hence the computed results will have a high variance.

We define the error $\delta_{k} \in \mathbb{R}^N$ of the Monte Carlo simulation, at time $t_k=k\Delta t$, using the elementwise notation
\begin{equation}
	\label{eq:MCerror}
     \delta_{n,k} =	\rho(x_n,t_k) - \hat{\rho}_{n,k}, \qquad n = 0,\dots,N-1, \qquad k\geq 0.
\end{equation}
As the discretization error $\delta_{k}$ is stochastic and independent for every realization, we use mean-squared error (MSE) as an error metric. For notational simplicity, we do so at the level of individual cells. We decompose the error into the squared bias and variance, i.e.,
\begin{equation}
	\label{eq:errordecomposition}
	\mathbb{E}[\delta_{n,k}^2] = \mathbb{E}^2[\delta_{n,k}] + \mathbb{V}[\delta_{n,k}].
\end{equation}

\subsubsection{Bias.} For the homogeneous diffusion equation, the discretized dynamics \eqref{eq:EulerMaruyama} produced unbiased trajectories corresponding with realizations of the SDE \eqref{eq:SDE}, i.e., the time discretization does not introduce an error. We remark, however, that the discretization of more general SDEs will result in biased approximations. Here, the only sources of error lie in the estimation step. Specifically, \eqref{eq:binning} is a biased estimator, due to producing a histogram with bins of width $\Delta x$. By construction, the expectation of \eqref{eq:binning} is given by
\begin{equation}
	\mathbb{E}\left[\hat{\rho}_{n,k}\right] = \frac{1}{\Delta x} \int_{n\Delta x}^{(n+1)\Delta x} \rho(x,t_k) \text{d}x = \frac{F((n+1)\Delta x, t_k) - F(n\Delta x, t_k)}{\Delta x},
\end{equation} 
with $F(x,t)$ the primitive of $\rho(x,t)$, i.e., the cumulative distribution function of $X(t)$. Taking the Taylor series around an arbitrary point $x\in[n\Delta x, (n+1)\Delta x]$, we get that
\begin{equation}
	\mathbb{E}\left[\hat{\rho}_{n,k}\right] = \rho(x,t_k) + \left( \left( n + \frac{1}{2} \right) \Delta x - x \right)\frac{\partial \rho}{\partial x}(x,t_k) + \mathcal{O}(\Delta x^2).
\end{equation}
Hence, the values at the $x_n$ have a binning bias that decays with order two in $\Delta x\to 0$. For general values of $x$, we note that the function is approximated with order one in $\Delta x$ as 
\begin{equation}
 \left|	\left( n + \frac{1}{2} \right) \Delta x - x \right| \leq \frac{\Delta x}{2}.
\end{equation}

\subsubsection{Variance.} The finite ensemble size $P$ induces a sampling error that scales with the number of particles in a given cell. As each particle in the ensemble \eqref{eq:particleensemble} is fully independent and produces a single contribution to the binning estimator \eqref{eq:binning}, we observe by the central limit theorem that the variance of the individual grid cells scales with $\mathcal{O}\left( \frac{1}{P\Delta x} \right)$ as $P\Delta x \to \infty$. Moreover, the distribution of $\delta_{k}$ will converge to a multivariate Gaussian in this limit, with mean and variance corresponding to the bias and variance in the error decomposition \eqref{eq:errordecomposition}. Hence, motivating the assumed distribution in \eqref{eq:randomForwardMap}. Despite the trajectory independence, we note that the grid cells themselves are correlated due to the fact that each trajectory can only score in one cell at a given time step.

\section{Analysis of approximate likelihood ratios}
\label{sec:MCMC}

We now take a deeper dive into how a stochastic approximation error $\delta \in \mathbb{R}^N$, as specified in \eqref{eq:randomForwardMap}, affects ratios of approximate likelihood evaluations. In what follows, we drop the subscripts used in Section~\ref{sec:mcfordiffusion}, for notational simplicity. Throughout this section we make use of inner products and norms induced by symmetric strictly positive definite matrices. For example, given two arbitrary vectors $\alpha, \beta \in \mathbb{R}^N$ and a symmetric strictly positive definite matrix $Q = Q^\transp \in \mathbb{R}^{N\times N}_{\succ 0}$, we have that $\left\langle \alpha, \beta \right\rangle_Q = \alpha^\transp Q \beta$ and $\left\Vert \alpha\right\Vert_Q^2 = \alpha^\transp Q \alpha$.

Given a discretization error \eqref{eq:randomForwardMap}, we introduce an approximate likelihood function by plugging the difference $\rho-\hat{G}[D]$ into the PDF of the normal distribution modeling the additive observation error in \eqref{eq:inverseProblem}, i.e.,
\begin{equation}
\hat{l}(D) = \frac{1}{\sqrt{\left(2\pi\right)^N|\Sigma_\eta|}}\exp\left(-\frac{1}{2}\left\Vert \rho-\hat{G}[D]  \right\Vert^2_{\Sigma_\eta^{-1}}\right).
	\label{eq:likelihoodfunction}
\end{equation}

We now present this paper's core theoretical result, relating the covariance matrices of the distributions of the random variables $\delta$ and $\eta$ to the existence of moments of the distribution of ratios of approximate likelihoods \eqref{eq:likelihoodfunction}.

\begin{theorem}
	\label{thm:moments}
	Given two parameter values $D_1, D_2 \in \mathbb{R}$ and an approximate forward map $\hat{G}: \mathbb{R} \mapsto \mathbb{R}^N$ such that $\hat{G}[D] = G[D] + \delta$ with $\delta \overset{\text{i.i.d.}}{\sim} \mathcal{N}(\mu(D), \Sigma(D))$ where $\mu: \mathbb{R} \mapsto \mathbb{R}^N$ and $\Sigma = \Sigma^\transp: \mathbb{R} \mapsto \mathbb{R}^{N \times N}_{\succ 0}$. The approximate likelihood ratio under the assumption of Gaussian observation error \eqref{eq:inverseProblem}, denoted by
	\begin{equation}
		\label{eq:likelihoodratio}
	\frac{\hat{l}(D_1)}{\hat{l}(D_2)},
	\end{equation}
	is a random variable, whose $p$-th moment exists iff $p\Sigma_\eta \succ \Sigma(D_2)$, with $\succ$ denoting the Löwner order.
\end{theorem}
\begin{proof}
	We introduce notation $\mu_i = \mu(D_i)$, $\Sigma_i = \Sigma(D_i)$, $\Delta\rho_i = \rho - G[D_i]$ and $\delta_i \sim \mathcal{N}(\mu_i, \Sigma_i)$ for $i=1,2$. Subsequently, we write
	\begin{equation}
		\label{eq:relativelikelihooderror}
		\frac{\hat{l}(D_1)}{\hat{l}(D_2)} =  \exp\left( \frac{\left\Vert \Delta\rho_2 - \delta_2 \right\Vert^2_{\Sigma_\eta^{-1}} - \left\Vert \Delta\rho_1 - \delta_1 \right\Vert^2_{\Sigma_\eta^{-1}}}{2}  \right),
	\end{equation}
	taking into account that the normalizing constants in \eqref{eq:likelihoodfunction} are independent of $\delta_1$ and $\delta_2$.
	As noted in Section~\ref{sec:mcfordiffusion}, $\delta_1$ and $\delta_2$ are independent. Hence, the $p$-th raw moment of \eqref{eq:relativelikelihooderror} is given by the product
	\begin{equation}
		\label{eq:productofexpectationsp}
		\mathbb{E}_{\delta_1}\left[ \exp\left( - p \frac{\left\Vert \Delta\rho_1 - \delta_1 \right\Vert^2_{\Sigma_\eta^{-1}}}{2} \right) \right] \mathbb{E}_{\delta_2}\left[ \exp\left( p \frac{\left\Vert \Delta\rho_2 - \delta_2 \right\Vert^2_{\Sigma_\eta^{-1}}}{2} \right) \right],
	\end{equation}
	which we rewrite using the substitution $M^{-1} = p\Sigma_\eta^{-1}$ to get
	\begin{equation}
		\label{eq:productofexpectations}
		\mathbb{E}_{\delta_1}\left[ \exp\left( - \frac{\left\Vert \delta_1 -  \Delta\rho_1 \right\Vert^2_{M^{-1}}}{2} \right) \right] \mathbb{E}_{\delta_2}\left[ \exp\left( \frac{\left\Vert \delta_2 - \Delta\rho_2 \right\Vert^2_{M^{-1}}}{2} \right) \right].
	\end{equation}
	
	We separately work out the two expectations. The first factor of \eqref{eq:productofexpectations} becomes 
	\begin{equation}
		\label{eq:expectation1}
		 \bigintsss_{\mathbb{R}^N}\!\! \exp\left( - \frac{\left\Vert \delta_1 -  \Delta\rho_1 \right\Vert^2_{M^{-1}}}{2} \right)\frac{1}{\sqrt{|2\pi\Sigma_1|}} \exp \left( - \frac{\left\Vert \delta_1-\mu_1 \right\Vert^2_{\Sigma_1^{-1}}}{2} \right) \text{d}\delta_1,
	\end{equation}
	 which we interpret as a product of Gaussian PDFs for $\delta_1$. We now show that this product is proportional to a single Gaussian PDF with precision matrix $S_1^{-1} = {M^{-1}}+\Sigma_1^{-1}$. Note that $S_1^{-1}$ is strictly positive definite as the sum of strictly positive definite matrices. Hence, we introduce $\alpha_1 = S_1 \left( {M^{-1}}\Delta \rho_1 + \Sigma_1^{-1} \mu_1 \right)$ as the mean of the Gaussion proportional to the integrand and work out that
	\begin{align}
		&\left\Vert \delta_1 -  \Delta\rho_1 \right\Vert^2_{M^{-1}} + \left\Vert \delta_1-\mu_1 \right\Vert^2_{\Sigma_1^{-1}} \\
		& \quad = \left\Vert \delta_1 \right\Vert^2_{M^{-1}} - 2\left\langle \delta_1, \Delta \rho_1 \right\rangle_{M^{-1}} + \left\Vert \Delta \rho_1 \right\Vert^2_{M^{-1}} + \left\Vert \delta_1 \right\Vert^2_{\Sigma_1^{-1}} - 2\left\langle \delta_1, \mu_1 \right\rangle_{\Sigma_1^{-1}} + \left\Vert \mu_1 \right\Vert^2_{\Sigma_1^{-1}} \\
		& \quad = \left\Vert \delta_1 \right\Vert^2_{S_1^{-1}} - 2 \left\langle \delta_1, \alpha_1 \right\rangle_{S_1^{-1}} + \left( \left\Vert\alpha_1\right\Vert^2_{S_1^{-1}} - \left\Vert\alpha_1\right\Vert^2_{S_1^{-1}} \right) + \left\Vert\Delta \rho_1\right\Vert^2_{M^{-1}} + \left\Vert\mu_1\right\Vert^2_{\Sigma_1^{-1}} \\
		& \quad = \left\Vert \delta_1 - \alpha_1 \right\Vert^2_{S_1^{-1}} - \left\Vert\alpha_1\right\Vert^2_{S_1^{-1}} + \left\Vert\Delta \rho_1\right\Vert^2_{M^{-1}} + \left\Vert\mu_1\right\Vert^2_{\Sigma_1^{-1}} \\
		& \quad \triangleq \left\Vert \delta_1 - \alpha_1 \right\Vert^2_{S_1^{-1}} + 2C_1,
	\end{align}
	where we use $\triangleq$ to denote that we have defined the $\delta_1$-independent constant $C_1$.
	
	We now write \eqref{eq:expectation1} as
	\begin{align}
		&\bigintsss_{\mathbb{R}^N}\!\! \exp\left( - \frac{\left\Vert \delta_1 -  \alpha_1 \right\Vert^2_{S_1^{-1}}}{2} \right)\frac{1}{\sqrt{|2\pi\Sigma_1|}} \exp \left( - C_1 \right) \text{d}\delta_1 \\
		&= \bigintsss_{\mathbb{R}^N}\!\! \exp\left( - \frac{\left\Vert \delta_1 -  \alpha_1 \right\Vert^2_{S_1^{-1}}}{2} \right)\frac{1}{\sqrt{|2\pi\Sigma_1|}} \exp \left( - C_1 \right) \text{d}\delta_1 \\
		&= \sqrt{\frac{|S_1|}{|\Sigma_1|}}\exp\left( - C_1\right) \label{eq:productresult}.
	\end{align}

	The second factor in \eqref{eq:productofexpectations} requires a bit more care due to the differing signs. We first write the factor as
	\begin{equation}
		\label{eq:expectation2}
		\bigintsss_{\mathbb{R}^N}\!\! \exp\left(\frac{\left\Vert \delta_2 -  \Delta\rho_2 \right\Vert^2_{M^{-1}}}{2} \right)\frac{1}{\sqrt{|2\pi \Sigma_2|}} \exp \left( - \frac{\left\Vert \delta_2-\mu_2 \right\Vert^2_{\Sigma_2^{-1}}}{2} \right) \text{d}\delta_2.
	\end{equation}
	 We define $S_2^\dagger = \Sigma_2^{-1}-{M^{-1}}$ and $\gamma_2 = \left(\Sigma_2^{-1} \mu_2 - {M^{-1}}\Delta \rho_2 \right)$. Note that we consider $S_2^\dagger$ as the Moore-Penrose pseudoinverse of a not yet introduced matrix $S_2$, as it may be singular. We write $\delta_2 = \delta_{2,0} + \delta_{2,\perp}$ and $\gamma_2 = \gamma_{2,0} + \gamma_{2,\perp}$, with $\delta_{2,0}, \gamma_{2,0} \in \ker (S_2^\dagger)$ and $\delta_{2,\perp}, \gamma_{2,\perp}\in \text{range}(S_2^\dagger)$. We then define $\alpha_2 = S_2 \gamma_2 = S_2 \gamma_{2,\perp}$. For compactness, we allow some abuse of notation, still using norms and inner product notation weighted by the potentially indefinite matrix $S_2^\dagger$. 
	 
	 Following a similar argument to before, and making use of the fact that $\delta_{2,0} \perp \gamma_{2,\perp}$, we get
	\begin{align}
		&\left\Vert \delta_2-\mu_2 \right\Vert^2_{\Sigma_2^{-1}} - \left\Vert \delta_2 -  \Delta\rho_2 \right\Vert^2_{M^{-1}}\\
		& =  \left\Vert \delta_2 \right\Vert^2_{\Sigma_2^{-1}} - 2\left\langle \delta_2, \mu_2 \right\rangle_{\Sigma_2^{-1}} + \left\Vert \mu_2 \right\Vert^2_{\Sigma_2^{-1}} - \left\Vert \delta_2 \right\Vert^2_{M^{-1}} + 2\left\langle \delta_2, \Delta \rho_2 \right\rangle_{M^{-1}} - \left\Vert \Delta \rho_2 \right\Vert^2_{M^{-1}}\\
		& = \left\Vert \delta_{2,\perp}\right\Vert^2_{S_2^\dagger} - 2 \left\langle \delta_{2,\perp}, \alpha_2 \right\rangle_{S_2^\dagger} - 2 \left \langle \delta_{2,0}, \gamma_{2,0} \right\rangle + \left\Vert \mu_2 \right\Vert^2_{\Sigma_2^{-1}} - \left\Vert \Delta \rho_2 \right\Vert^2_{M^{-1}} \\
		&= \left\Vert \delta_{2,\perp} - \alpha_2 \right\Vert^2_{S_2^\dagger}  - 2 \left \langle \delta_{2,0}, \gamma_{2,0} \right\rangle  - \left\Vert \alpha_2 \right\Vert^2_{S_2^\dagger} +  \left\Vert \mu_2 \right\Vert^2_{\Sigma_2^{-1}} - \left\Vert \Delta \rho_2 \right\Vert^2_{M^{-1}} \\
		& \triangleq \left\Vert \delta_{2,\perp} - \alpha_2 \right\Vert^2_{S_2^\dagger} - 2 \left \langle \delta_{2,0}, \gamma_{2,0} \right\rangle   + 2C_2
	\end{align}
	where, similar to $C_1$, $C_2$ is a constant term, independent of $\delta_2$.
	
	We now work out \eqref{eq:expectation2} as 
	\begin{equation}
		\bigintsss_{\mathbb{R}^N}\!\! \exp\left( - \frac{\left\Vert \delta_{2,\perp} -  \alpha_2 \right\Vert^2_{S_2^\dagger}}{2} \right) \exp \left( \left \langle \delta_{2,0}, \gamma_{2,0} \right\rangle  \right) \frac{1}{\sqrt{|2\pi \Sigma_2|}} \exp \left( - C_2 \right) \text{d}\delta_2 \label{eq:ratiointegral}.
	\end{equation}
	We observe that the integrand is a smooth function that only vanishes at infinity if (i) $\delta_{2,0}$ remains bounded and (ii) the nonzero eigenvalues of $S_2^\dagger$ are positive. Condition (i) can only be met if $\delta_{2,0} = 0$ and $S_2^\dagger$ must be invertible and thus strictly positive definite. We therefore state $\forall \delta_2 \in \mathbb{R}^N: \delta_2 = \delta_{2,\perp}$. Hence, \eqref{eq:ratiointegral} evaluates to
	\begin{equation}
		\sqrt{\frac{|S_2|}{|\Sigma_2|}}\exp\left( - C_2\right) \label{eq:ratioresult}.
	\end{equation}
	
	We combine integrability condition (ii) for \eqref{eq:ratiointegral} with the fact that \eqref{eq:expectation1} is strictly positive and the definition of $M$ to conclude that the $p$-th moment only exists iff $\Sigma_2^{-1} \succ p\Sigma_\eta^{-1}$. From the definition of the Löwner order, one can show that this condition is equivalent to $p\Sigma_\eta \succ \Sigma(D_2)$.
	 \qed
\end{proof}

\begin{remark}
	During the proof of Theorem~\ref{thm:moments}, we make extensive use of the properties of normal distributions to compute bounds for the moments of the distribution of the approximate likelihood ratio \eqref{eq:likelihoodratio}. However, we note that this ratio itself will not be normally distributed as it only has finitely many moments.
\end{remark}

\begin{remark}
	The proof of Theorem~\ref{thm:moments} is quite verbose, in order to derive explicit expressions for the $p$-th moment of the distribution of the ratio \eqref{eq:likelihoodratio}. One can follow a much shorter argument to state that there exists an upper bound for $\Sigma(D_2)$ as a function of the number of desired moments, for example by observing that the numerator of \eqref{eq:likelihoodratio} is bounded and applying Fernique's theorem~\cite{DaPrato2014} to the denominator.
\end{remark}

\section{Numerical experiments}
\label{sec:numerics}

We now perform a numerical experiment to confirm the observations in Theorem~\ref{thm:moments}. To this end, we use the model problem \eqref{eq:diff_eq}--\eqref{eq:diff_eq_boundary} and consider transitions between fixed values for $D_1$ and $D_2$. We produce synthetic observations, using a finite difference discretization whose solution is perturbed with noise distributed according to \eqref{eq:inverseProblem}. Throughout this section we set $N=100$, $L=10$, $t=10$ and $D^\ast=0.1$. The reference solution to produce the observation is computed with $\Delta t = 0.1$, while the Monte Carlo simulations use $\Delta t = t$ to cut on computational cost, thanks to the unbiasedness of \eqref{eq:EulerMaruyama}. For $D_1$ and $D_2$, we consider differing combinations of these parameters from the set $\{0.08, 0.1\}$, i.e., the true value and a similar, but less probable, value. The code used to generate the results in this section can be found at \url{github.com/UQatKIT/FrontUQ-2024-proceedings-likelihood-ratios}.

To enable visualization, we simplify the covariance matrices of the distributions of $\eta$ and $\delta$ as follows. We fix $\Sigma_\eta = \sigma_\eta^2 I$, with $I$ the identity matrix. We simplify the Monte Carlo covariance to a scalar variance $\sigma_\delta^2$, by taking the maximum over all histogram cells, i.e.,
\begin{equation}
	\label{eq:sigmadelta}
	\sigma_\delta^2=\max(\text{diag}(\Sigma(D))),
\end{equation}
for the given value $D$. Taking the maximum is motivated by the resulting value being robust to the effects of bins with zero or very few particles. All expectations in this section are computed using Monte Carlo sampling with 1000 random realizations of the observation $\rho$. To generate each curve in the plots in this section, we take $P=10^2,\dots,10^6$ increasing in multiples of 10 and estimating $\sigma_\delta$ as specified in \eqref{eq:sigmadelta}. In some cases, in particular for small $P$, we get invalid results for the quantities being plotted. These invalid results result from the exponential in \eqref{eq:likelihoodfunction} evaluating to zero due to highly unlikely solver outputs. Taking the ratio of expectations then produces NaN values. As both zero and NaN values cannot be sensibly plotted in a log scale, these points are omitted from the figures.

Before considering expectations of ratios of approximate likelihoods, we first present the expectations of the approximate likelihoods themselves in Figure~\ref{fig:likelihoods}. We  observe that the expected approximate likelihood for $D_1=0.1$ (Figure~\ref{fig:likelystate}) is consistently larger than that for $D_2=0.08$ (Figure~\ref{fig:unlikelystate}), as can be seen in Figure~\ref{fig:ratioofexpectations} that shows the ratio of the expectations of these approximate likelihoods. We also note the sharp drop in the expectation of the approximate likelihood for both $D_1$ and $D_2$ once $\sigma_\eta<\sigma_\delta$.  As one would expect, for larger values of $P$, the ratio of expectations in Figure~\ref{fig:ratioofexpectations} converges to one as $\sigma_\eta$ increases.

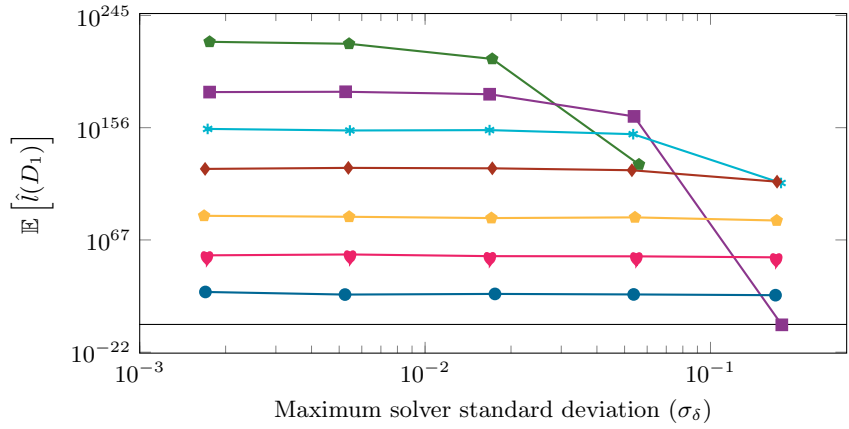
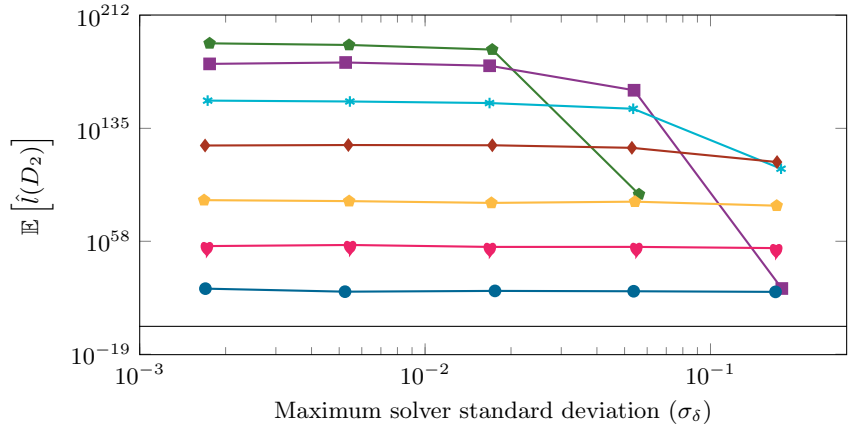
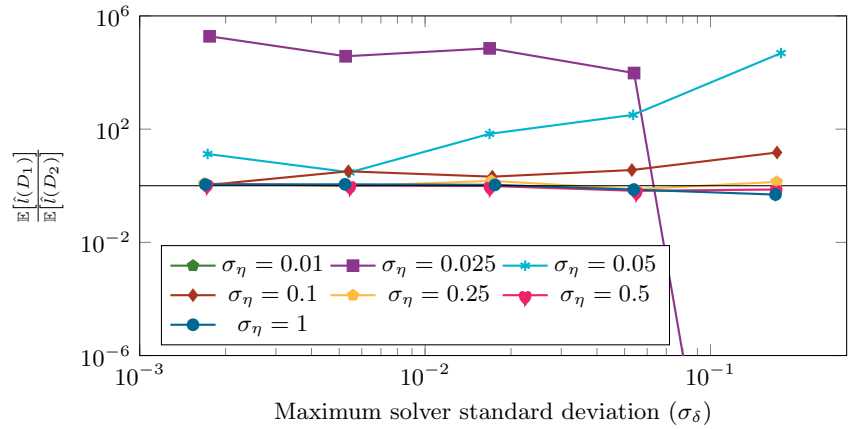
\begin{figure}
	\centering
	\begin{subfigure}{\linewidth}
		\centering
		\begin{tikzpicture}[trim axis left]
			\begin{axis}[
				xlabel={Maximum solver standard deviation ($\sigma_\delta$)},
				ylabel={$\mathbb{E}\left[\hat{l}(D_1) \right]$},
				xmode=log,
				ymode=log,
				xmin=1e-3,
				xmax=0.3,
				legend pos=south west,
				legend columns=3,
				width=0.9\textwidth,
				height=0.5\textwidth
				]
				\addplot[
				color=OliveGreen,
				mark=pentagon*,
				thick
				]
				table[
				col sep=comma,
				x=maxsolverstdsigma0.01,
				y=meansigma0.01,
				] {numerical_results/plotdata_proposal_likelihood_true_0.1_current_0.08_proposal_0.1_samples_1000.csv};
				\addlegendentry{$\sigma_\eta=0.01$}
				\addplot[
				color=Fuchsia,
				mark=square*,
				thick
				]
				table[
				col sep=comma,
				x=maxsolverstdsigma0.025,
				y=meansigma0.025,
				] {numerical_results/plotdata_proposal_likelihood_true_0.1_current_0.08_proposal_0.1_samples_1000.csv};
				\addlegendentry{$\sigma_\eta=0.025$}
				\addplot[
				color=Turquoise,
				mark=asterisk,
				thick
				]
				table[
				col sep=comma,
				x=maxsolverstdsigma0.05,
				y=meansigma0.05,
				] {numerical_results/plotdata_proposal_likelihood_true_0.1_current_0.08_proposal_0.1_samples_1000.csv};
				\addlegendentry{$\sigma_\eta=0.05$}
				\addplot[
				color=Mahogany,
				mark=diamond*,
				thick
				]
				table[
				col sep=comma,
				x=maxsolverstdsigma0.1,
				y=meansigma0.1,
				] {numerical_results/plotdata_proposal_likelihood_true_0.1_current_0.08_proposal_0.1_samples_1000.csv};
				\addlegendentry{$\sigma_\eta=0.1$}
				\addplot[
				color=Dandelion,
				mark=pentagon*,
				thick
				]
				table[
				col sep=comma,
				x=maxsolverstdsigma0.25,
				y=meansigma0.25,
				] {numerical_results/plotdata_proposal_likelihood_true_0.1_current_0.08_proposal_0.1_samples_1000.csv};
				\addlegendentry{$\sigma_\eta=0.25$}
				\addplot[
				color=WildStrawberry,
				mark=heart,
				thick
				]
				table[
				col sep=comma,
				x=maxsolverstdsigma0.5,
				y=meansigma0.5,
				] {numerical_results/plotdata_proposal_likelihood_true_0.1_current_0.08_proposal_0.1_samples_1000.csv};
				\addlegendentry{$\sigma_\eta=0.5$}
				\addplot[
				color=MidnightBlue,
				mark=*,
				thick
				]
				table[
				col sep=comma,
				x=maxsolverstdsigma1,
				y=meansigma1,
				] {numerical_results/plotdata_proposal_likelihood_true_0.1_current_0.08_proposal_0.1_samples_1000.csv};
				\addlegendentry{$\sigma_\eta=1$}
				\addplot[color=black, mark=none, domain=0.5:2e-4]{1};
				\legend{};
			\end{axis}
		\end{tikzpicture}
		\caption{Expectation of the proposal likelihood at $D_1=0.1$. \label{fig:likelystate}}
	\end{subfigure}
	\par\bigskip
	\begin{subfigure}{\linewidth}
		\centering
		\begin{tikzpicture}[trim axis left]
			\begin{axis}[
				xlabel={Maximum solver standard deviation ($\sigma_\delta$)},
				ylabel={$\mathbb{E}\left[\hat{l}(D_2) \right]$},
				xmode=log,
				ymode=log,
				xmin=1e-3,
				xmax=0.3,
				legend pos=north east,
				width=0.9\textwidth,
				height=0.5\textwidth
				]
				\addplot[
				color=OliveGreen,
				mark=pentagon*,
				thick
				]
				table[
				col sep=comma,
				x=maxsolverstdsigma0.01,
				y=meansigma0.01,
				] {numerical_results/plotdata_current_likelihood_true_0.1_current_0.08_proposal_0.1_samples_1000.csv};
				\addplot[
				color=Fuchsia,
				mark=square*,
				thick
				]
				table[
				col sep=comma,
				x=maxsolverstdsigma0.025,
				y=meansigma0.025,
				] {numerical_results/plotdata_current_likelihood_true_0.1_current_0.08_proposal_0.1_samples_1000.csv};
				\addplot[
				color=Turquoise,
				mark=asterisk,
				thick
				]
				table[
				col sep=comma,
				x=maxsolverstdsigma0.05,
				y=meansigma0.05,
				] {numerical_results/plotdata_current_likelihood_true_0.1_current_0.08_proposal_0.1_samples_1000.csv};
				\addplot[
				color=Mahogany,
				mark=diamond*,
				thick
				]
				table[
				col sep=comma,
				x=maxsolverstdsigma0.1,
				y=meansigma0.1,
				] {numerical_results/plotdata_current_likelihood_true_0.1_current_0.08_proposal_0.1_samples_1000.csv};
				\addplot[
				color=Dandelion,
				mark=pentagon*,
				thick
				]
				table[
				col sep=comma,
				x=maxsolverstdsigma0.25,
				y=meansigma0.25,
				] {numerical_results/plotdata_current_likelihood_true_0.1_current_0.08_proposal_0.1_samples_1000.csv};
				\addplot[
				color=WildStrawberry,
				mark=heart,
				thick
				]
				table[
				col sep=comma,
				x=maxsolverstdsigma0.5,
				y=meansigma0.5,
				] {numerical_results/plotdata_current_likelihood_true_0.1_current_0.08_proposal_0.1_samples_1000.csv};
				\addplot[
				color=MidnightBlue,
				mark=*,
				thick
				]
				table[
				col sep=comma,
				x=maxsolverstdsigma1,
				y=meansigma1,
				] {numerical_results/plotdata_current_likelihood_true_0.1_current_0.08_proposal_0.1_samples_1000.csv};
				\addplot[color=black, mark=none, domain=0.5:2e-4]{1};
			\end{axis}
		\end{tikzpicture}
		\caption{Expectation of the current state's likelihood at $D_2=0.08$. \label{fig:unlikelystate}}
	\end{subfigure}
	\par\bigskip
	\begin{subfigure}{\linewidth}
		\centering
		\begin{tikzpicture}[trim axis left]
			\begin{axis}[
				xlabel={Maximum solver standard deviation ($\sigma_\delta$)},
				ylabel={$\frac{\mathbb{E}\left[\hat{l}(D_1)\right]}{\mathbb{E}\left[\hat{l}(D_2)\right]}$},
				xmode=log,
				ymode=log,
				xmin=1e-3,
				xmax=0.3,
				ymax=1e6,
				ymin=1e-6,
				legend pos=south west,
				legend columns=3,
				width=0.9\textwidth,
				height=0.5\textwidth
				]
				\addplot[
				color=OliveGreen,
				mark=pentagon*,
				thick
				]
				table[
				col sep=comma,
				x=maxsolverstdsigma0.01,
				y=ratiosigma0.01,
				] {numerical_results/plotdata_ratio_expectations_true_0.1_current_0.08_proposal_0.1_samples_1000.csv};
				\addlegendentry{$\sigma_\eta=0.01$}
				\addplot[
				color=Fuchsia,
				mark=square*,
				thick
				]
				table[
				col sep=comma,
				x=maxsolverstdsigma0.025,
				y=ratiosigma0.025,
				] {numerical_results/plotdata_ratio_expectations_true_0.1_current_0.08_proposal_0.1_samples_1000.csv};
				\addlegendentry{$\sigma_\eta=0.025$}
				\addplot[
				color=Turquoise,
				mark=asterisk,
				thick
				]
				table[
				col sep=comma,
				x=maxsolverstdsigma0.05,
				y=ratiosigma0.05,
				] {numerical_results/plotdata_ratio_expectations_true_0.1_current_0.08_proposal_0.1_samples_1000.csv};
				\addlegendentry{$\sigma_\eta=0.05$}
				\addplot[
				color=Mahogany,
				mark=diamond*,
				thick
				]
				table[
				col sep=comma,
				x=maxsolverstdsigma0.1,
				y=ratiosigma0.1,
				] {numerical_results/plotdata_ratio_expectations_true_0.1_current_0.08_proposal_0.1_samples_1000.csv};
				\addlegendentry{$\sigma_\eta=0.1$}
				\addplot[
				color=Dandelion,
				mark=pentagon*,
				thick
				]
				table[
				col sep=comma,
				x=maxsolverstdsigma0.25,
				y=ratiosigma0.25,
				] {numerical_results/plotdata_ratio_expectations_true_0.1_current_0.08_proposal_0.1_samples_1000.csv};
				\addlegendentry{$\sigma_\eta=0.25$}
				\addplot[
				color=WildStrawberry,
				mark=heart,
				thick
				]
				table[
				col sep=comma,
				x=maxsolverstdsigma0.5,
				y=ratiosigma0.5,
				] {numerical_results/plotdata_ratio_expectations_true_0.1_current_0.08_proposal_0.1_samples_1000.csv};
				\addlegendentry{$\sigma_\eta=0.5$}
				\addplot[
				color=MidnightBlue,
				mark=*,
				thick
				]
				table[
				col sep=comma,
				x=maxsolverstdsigma1,
				y=ratiosigma1,
				] {numerical_results/plotdata_ratio_expectations_true_0.1_current_0.08_proposal_0.1_samples_1000.csv};
				\addlegendentry{$\sigma_\eta=1$}
				\addplot[color=black, mark=none, domain=0.5:2e-4]{1};
			\end{axis}
		\end{tikzpicture}
		\caption{Ratio of the expectations of the proposal and current state's likelihoods. The curve for $\sigma_\eta=0.01$ is not visible due to taking values between $10^{21}$ and $10^{37}$. \label{fig:ratioofexpectations}}
	\end{subfigure}
	\caption{Computing the expectation of the approximate likelihoods and the ratio of these expectations, given a discrete solution of the problem \eqref{eq:diff_eq}--\eqref{eq:diff_eq_boundary} perturbed by synthetic observation noise, for $D^\ast = D_1 =0.1$ and $D_2 = 0.08$. \label{fig:likelihoods}}
\end{figure}

\begin{figure}
	\centering
	\begin{subfigure}{\linewidth}
		\centering
		\begin{tikzpicture}[trim axis left]
			\begin{axis}[
				xlabel={Maximum solver standard deviation ($\sigma_\delta$)},
				ylabel={$\mathbb{E}\left[\frac{\hat{l}(D_1)}{\hat{l}(D_2)} \right]$},
				xmode=log,
				ymode=log,
				xmin=1e-3,
				xmax=0.3,
				ymax=1e35,
				legend pos=north west,
				legend columns=1,
				width=0.9\textwidth,
				height=0.5\textwidth
				]
				\addplot[
				color=OliveGreen,
				mark=pentagon*,
				thick
				]
				table[
				col sep=comma,
				x=maxsolverstdsigma0.01,
				y=meansigma0.01,
				] {numerical_results/plotdata_likelihood_ratio_true_0.1_current_0.1_proposal_0.1_samples_1000.csv};
				\addlegendentry{$\sigma_\eta=0.01$}
				\addplot[
				color=Fuchsia,
				mark=square*,
				thick
				]
				table[
				col sep=comma,
				x=maxsolverstdsigma0.025,
				y=meansigma0.025,
				] {numerical_results/plotdata_likelihood_ratio_true_0.1_current_0.1_proposal_0.1_samples_1000.csv};
				\addlegendentry{$\sigma_\eta=0.025$}
				\addplot[
				color=Turquoise,
				mark=asterisk,
				thick
				]
				table[
				col sep=comma,
				x=maxsolverstdsigma0.05,
				y=meansigma0.05,
				] {numerical_results/plotdata_likelihood_ratio_true_0.1_current_0.1_proposal_0.1_samples_1000.csv};
				\addlegendentry{$\sigma_\eta=0.05$}
				\addplot[
				color=Mahogany,
				mark=diamond*,
				thick
				]
				table[
				col sep=comma,
				x=maxsolverstdsigma0.1,
				y=meansigma0.1,
				] {numerical_results/plotdata_likelihood_ratio_true_0.1_current_0.1_proposal_0.1_samples_1000.csv};	
				\addlegendentry{$\sigma_\eta=0.1$}
				\addplot[
				color=Dandelion,
				mark=pentagon*,
				thick
				]
				table[
				col sep=comma,
				x=maxsolverstdsigma0.25,
				y=meansigma0.25,
				] {numerical_results/plotdata_likelihood_ratio_true_0.1_current_0.1_proposal_0.1_samples_1000.csv};
				\addlegendentry{$\sigma_\eta=0.25$}
				\addplot[
				color=WildStrawberry,
				mark=heart,
				thick
				]
				table[
				col sep=comma,
				x=maxsolverstdsigma0.5,
				y=meansigma0.5,
				] {numerical_results/plotdata_likelihood_ratio_true_0.1_current_0.1_proposal_0.1_samples_1000.csv};
				\addlegendentry{$\sigma_\eta=0.5$}
				\addplot[
				color=MidnightBlue,
				mark=*,
				thick
				]
				table[
				col sep=comma,
				x=maxsolverstdsigma1,
				y=meansigma1,
				] {numerical_results/plotdata_likelihood_ratio_true_0.1_current_0.1_proposal_0.1_samples_1000.csv};
				\addlegendentry{$\sigma_\eta=1$}
				\addplot[color=black, mark=none, domain=0.5:2e-4]{1};
			\end{axis}
		\end{tikzpicture}
		\caption{Expectation of the approximate likelihood ratio.\label{fig:likelytolikelyratio}}
	\end{subfigure}
	\par\bigskip
	\begin{subfigure}{\linewidth}
		\centering
		\begin{tikzpicture}[trim axis left]
			\begin{axis}[
				xlabel={Maximum solver standard deviation ($\sigma_\delta$)},
				ylabel={$\mathbb{E}\left[\min \left(\frac{\hat{l}(D_1)}{\hat{l}(D_2)},1\right) \right]$},
				xmode=log,
				ymax=1.0,
				ymin=-0.0,
				xmin=1e-3,
				xmax=0.3,
				width=0.9\textwidth,
				height=0.5\textwidth
				]
				\addplot[
				color=OliveGreen,
				mark=pentagon*,
				thick
				]
				table[
				col sep=comma,
				x=maxsolverstdsigma0.01,
				y=meansigma0.01,
				] {numerical_results/plotdata_acceptance_rate_true_0.1_current_0.1_proposal_0.1_samples_1000.csv};
				\addlegendentry{$\sigma_\eta=0.01$}
				\addplot[
				color=Fuchsia,
				mark=square*,
				thick
				]
				table[
				col sep=comma,
				x=maxsolverstdsigma0.025,
				y=meansigma0.025,
				] {numerical_results/plotdata_acceptance_rate_true_0.1_current_0.1_proposal_0.1_samples_1000.csv};
				\addlegendentry{$\sigma_\eta=0.025$}
				\addplot[
				color=Turquoise,
				mark=asterisk,
				thick
				]
				table[
				col sep=comma,
				x=maxsolverstdsigma0.05,
				y=meansigma0.05,
				] {numerical_results/plotdata_acceptance_rate_true_0.1_current_0.1_proposal_0.1_samples_1000.csv};
				\addlegendentry{$\sigma_\eta=0.05$}
				\addplot[
				color=Mahogany,
				mark=diamond*,
				thick
				]
				table[
				col sep=comma,
				x=maxsolverstdsigma0.1,
				y=meansigma0.1,
				] {numerical_results/plotdata_acceptance_rate_true_0.1_current_0.1_proposal_0.1_samples_1000.csv};	
				\addlegendentry{$\sigma_\eta=0.1$}
				\addplot[
				color=Dandelion,
				mark=pentagon*,
				thick
				]
				table[
				col sep=comma,
				x=maxsolverstdsigma0.25,
				y=meansigma0.25,
				] {numerical_results/plotdata_acceptance_rate_true_0.1_current_0.1_proposal_0.1_samples_1000.csv};
				\addlegendentry{$\sigma_\eta=0.25$}
				\addplot[
				color=WildStrawberry,
				mark=heart,
				thick
				]
				table[
				col sep=comma,
				x=maxsolverstdsigma0.5,
				y=meansigma0.5,
				] {numerical_results/plotdata_acceptance_rate_true_0.1_current_0.1_proposal_0.1_samples_1000.csv};
				\addlegendentry{$\sigma_\eta=0.5$}
				\addplot[
				color=MidnightBlue,
				mark=*,
				thick
				]
				table[
				col sep=comma,
				x=maxsolverstdsigma1,
				y=meansigma1,
				] {numerical_results/plotdata_acceptance_rate_true_0.1_current_0.1_proposal_0.1_samples_1000.csv};
				\addlegendentry{$\sigma_\eta=1$}
				\addplot[color=black, mark=none, domain=0.5:2e-4]{0};
				\addplot[color=black, mark=none, domain=0.5:2e-4]{1};
				\legend{}
			\end{axis}
		\end{tikzpicture}
		\caption{Expectation of the approximate likelihood ratio, truncated to the range $[0,1]$. \label{fig:likelytolikelyacceptance}}
	\end{subfigure}
	\caption{Computing the expectation of the approximate likelihood ratio given a discrete solution of the problem \eqref{eq:diff_eq}--\eqref{eq:diff_eq_boundary} perturbed by synthetic observation noise for $D^\ast = D_1 = D_2 =0.1$. \label{fig:likelytolikely}}
\end{figure}
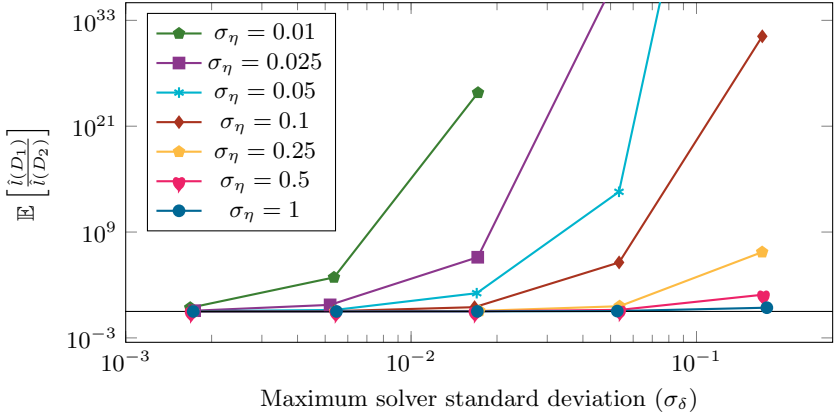
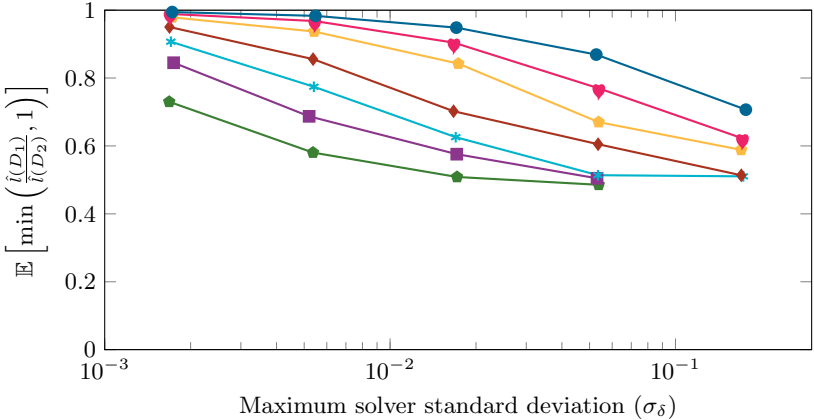

\begin{figure}
	\centering
	\begin{subfigure}{\linewidth}
		\centering
		\begin{tikzpicture}[trim axis left]
			\begin{axis}[
				xlabel={Maximum solver standard deviation ($\sigma_\delta$)},
				ylabel={$\mathbb{E}\left[\frac{\hat{l}(D_1)}{\hat{l}(D_2)} \right]$},
				xmode=log,
				ymode=log,
				xmin=1e-3,
				xmax=0.3,
				ymax=1e35,
				legend pos=north west,
				legend columns=2,
				width=0.9\textwidth,
				height=0.5\textwidth
				]
				\addplot[
				color=OliveGreen,
				mark=pentagon*,
				thick
				]
				table[
				col sep=comma,
				x=maxsolverstdsigma0.01,
				y=meansigma0.01,
				] {numerical_results/plotdata_likelihood_ratio_true_0.1_current_0.1_proposal_0.08_samples_1000.csv};
				\addlegendentry{$\sigma_\eta=0.01$}
				\addplot[
				color=Fuchsia,
				mark=square*,
				thick
				]
				table[
				col sep=comma,
				x=maxsolverstdsigma0.025,
				y=meansigma0.025,
				] {numerical_results/plotdata_likelihood_ratio_true_0.1_current_0.1_proposal_0.08_samples_1000.csv};
				\addlegendentry{$\sigma_\eta=0.025$}
				\addplot[
				color=Turquoise,
				mark=asterisk,
				thick
				]
				table[
				col sep=comma,
				x=maxsolverstdsigma0.05,
				y=meansigma0.05,
				] {numerical_results/plotdata_likelihood_ratio_true_0.1_current_0.1_proposal_0.08_samples_1000.csv};
				\addlegendentry{$\sigma_\eta=0.05$}
				\addplot[
				color=Mahogany,
				mark=diamond*,
				thick
				]
				table[
				col sep=comma,
				x=maxsolverstdsigma0.1,
				y=meansigma0.1,
				] {numerical_results/plotdata_likelihood_ratio_true_0.1_current_0.1_proposal_0.08_samples_1000.csv};	
				\addlegendentry{$\sigma_\eta=0.1$}
				\addplot[
				color=Dandelion,
				mark=pentagon*,
				thick
				]
				table[
				col sep=comma,
				x=maxsolverstdsigma0.25,
				y=meansigma0.25,
				] {numerical_results/plotdata_likelihood_ratio_true_0.1_current_0.1_proposal_0.08_samples_1000.csv};
				\addlegendentry{$\sigma_\eta=0.25$}
				\addplot[
				color=WildStrawberry,
				mark=heart,
				thick
				]
				table[
				col sep=comma,
				x=maxsolverstdsigma0.5,
				y=meansigma0.5,
				] {numerical_results/plotdata_likelihood_ratio_true_0.1_current_0.1_proposal_0.08_samples_1000.csv};
				\addlegendentry{$\sigma_\eta=0.5$}
				\addplot[
				color=MidnightBlue,
				mark=*,
				thick
				]
				table[
				col sep=comma,
				x=maxsolverstdsigma1,
				y=meansigma1,
				] {numerical_results/plotdata_likelihood_ratio_true_0.1_current_0.1_proposal_0.08_samples_1000.csv};
				\addlegendentry{$\sigma_\eta=1$}
				\addplot[color=black, mark=none, domain=0.5:2e-4]{1};
			\end{axis}
		\end{tikzpicture}
		\caption{Expectation of the approximate likelihood ratio. \label{fig:likelytounlikelyratio}}
	\end{subfigure}
	\par\bigskip
	\begin{subfigure}{\linewidth}
		\centering
		\begin{tikzpicture}[trim axis left]
			\begin{axis}[
				xlabel={Maximum solver standard deviation ($\sigma_\delta$)},
				ylabel={$\mathbb{E}\left[\min \left(\frac{\hat{l}(D_1)}{\hat{l}(D_2)},1\right) \right]$},
				xmode=log,
				ymax=1.0,
				ymin=-0.0,
				xmin=1e-3,
				xmax=0.3,
				legend pos=north east,
				width=0.9\textwidth,
				height=0.5\textwidth
				]
				\addplot[
				color=OliveGreen,
				mark=pentagon*,
				thick
				]
				table[
				col sep=comma,
				x=maxsolverstdsigma0.01,
				y=meansigma0.01,
				] {numerical_results/plotdata_acceptance_rate_true_0.1_current_0.1_proposal_0.08_samples_1000.csv};
				\addlegendentry{$\sigma_\eta=0.01$}
				\addplot[
				color=Fuchsia,
				mark=square*,
				thick
				]
				table[
				col sep=comma,
				x=maxsolverstdsigma0.025,
				y=meansigma0.025,
				] {numerical_results/plotdata_acceptance_rate_true_0.1_current_0.1_proposal_0.08_samples_1000.csv};
				\addlegendentry{$\sigma_\eta=0.025$}
				\addplot[
				color=Turquoise,
				mark=asterisk,
				thick
				]
				table[
				col sep=comma,
				x=maxsolverstdsigma0.05,
				y=meansigma0.05,
				] {numerical_results/plotdata_acceptance_rate_true_0.1_current_0.1_proposal_0.08_samples_1000.csv};
				\addlegendentry{$\sigma_\eta=0.05$}
				\addplot[
				color=Mahogany,
				mark=diamond*,
				thick
				]
				table[
				col sep=comma,
				x=maxsolverstdsigma0.1,
				y=meansigma0.1,
				] {numerical_results/plotdata_acceptance_rate_true_0.1_current_0.1_proposal_0.08_samples_1000.csv};	
				\addlegendentry{$\sigma_\eta=0.1$}
				\addplot[
				color=Dandelion,
				mark=pentagon*,
				thick
				]
				table[
				col sep=comma,
				x=maxsolverstdsigma0.25,
				y=meansigma0.25,
				] {numerical_results/plotdata_acceptance_rate_true_0.1_current_0.1_proposal_0.08_samples_1000.csv};
				\addlegendentry{$\sigma_\eta=0.25$}
				\addplot[
				color=WildStrawberry,
				mark=heart,
				thick
				]
				table[
				col sep=comma,
				x=maxsolverstdsigma0.5,
				y=meansigma0.5,
				] {numerical_results/plotdata_acceptance_rate_true_0.1_current_0.1_proposal_0.08_samples_1000.csv};
				\addlegendentry{$\sigma_\eta=0.5$}
				\addplot[
				color=MidnightBlue,
				mark=*,
				thick
				]
				table[
				col sep=comma,
				x=maxsolverstdsigma1,
				y=meansigma1,
				] {numerical_results/plotdata_acceptance_rate_true_0.1_current_0.1_proposal_0.08_samples_1000.csv};
				\addlegendentry{$\sigma_\eta=1$}
				\addplot[color=black, mark=none, domain=0.5:2e-4]{0};
				\addplot[color=black, mark=none, domain=0.5:2e-4]{1};
				\legend{}
			\end{axis}
		\end{tikzpicture}
		\caption{Expectation of the approximate likelihood ratio, truncated to the range $[0,1]$. \label{fig:likelytounlikelyacceptance}}
	\end{subfigure}
	\caption{Computing the expectation of the approximate likelihood ratio given a discrete solution of the problem \eqref{eq:diff_eq}--\eqref{eq:diff_eq_boundary} perturbed by synthetic observation noise for $D^\ast = D_2 =0.1$ and $D_1 = 0.08$. \label{fig:likelytounlikely}}
\end{figure}
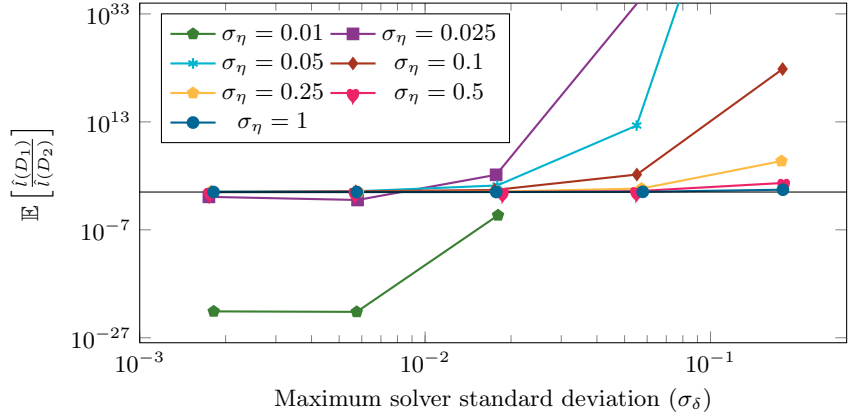
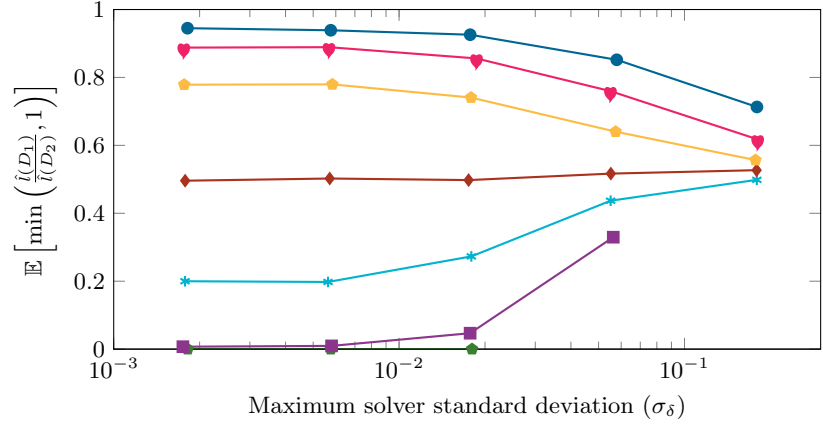

\begin{figure}
	\centering
	\begin{subfigure}{\linewidth}
		\centering
		\begin{tikzpicture}[trim axis left]
			\begin{axis}[
				xlabel={Maximum solver standard deviation ($\sigma_\delta$)},
				ylabel={$\mathbb{E}\left[\frac{\hat{l}(D_1)}{\hat{l}(D_2)} \right]$},
				xmode=log,
				ymode=log,
				xmin=1e-3,
				xmax=0.3,
				legend pos=north west,
				legend columns=2,
				width=0.9\textwidth,
				height=0.5\textwidth
				]
				\addplot[
				color=OliveGreen,
				mark=pentagon*,
				thick
				]
				table[
				col sep=comma,
				x=maxsolverstdsigma0.01,
				y=meansigma0.01,
				] {numerical_results/plotdata_likelihood_ratio_true_0.1_current_0.08_proposal_0.1_samples_1000.csv};
				\addlegendentry{$\sigma_\eta=0.01$}
				\addplot[
				color=Fuchsia,
				mark=square*,
				thick
				]
				table[
				col sep=comma,
				x=maxsolverstdsigma0.025,
				y=meansigma0.025,
				] {numerical_results/plotdata_likelihood_ratio_true_0.1_current_0.08_proposal_0.1_samples_1000.csv};
				\addlegendentry{$\sigma_\eta=0.025$}
				\addplot[
				color=Turquoise,
				mark=asterisk,
				thick
				]
				table[
				col sep=comma,
				x=maxsolverstdsigma0.05,
				y=meansigma0.05,
				] {numerical_results/plotdata_likelihood_ratio_true_0.1_current_0.08_proposal_0.1_samples_1000.csv};
				\addlegendentry{$\sigma_\eta=0.05$}
				\addplot[
				color=Mahogany,
				mark=diamond*,
				thick
				]
				table[
				col sep=comma,
				x=maxsolverstdsigma0.1,
				y=meansigma0.1,
				] {numerical_results/plotdata_likelihood_ratio_true_0.1_current_0.08_proposal_0.1_samples_1000.csv};	
				\addlegendentry{$\sigma_\eta=0.1$}
				\addplot[
				color=Dandelion,
				mark=pentagon*,
				thick
				]
				table[
				col sep=comma,
				x=maxsolverstdsigma0.25,
				y=meansigma0.25,
				] {numerical_results/plotdata_likelihood_ratio_true_0.1_current_0.08_proposal_0.1_samples_1000.csv};
				\addlegendentry{$\sigma_\eta=0.25$}
				\addplot[
				color=WildStrawberry,
				mark=heart,
				thick
				]
				table[
				col sep=comma,
				x=maxsolverstdsigma0.5,
				y=meansigma0.5,
				] {numerical_results/plotdata_likelihood_ratio_true_0.1_current_0.08_proposal_0.1_samples_1000.csv};
				\addlegendentry{$\sigma_\eta=0.5$}
				\addplot[
				color=MidnightBlue,
				mark=*,
				thick
				]
				table[
				col sep=comma,
				x=maxsolverstdsigma1,
				y=meansigma1,
				] {numerical_results/plotdata_likelihood_ratio_true_0.1_current_0.08_proposal_0.1_samples_1000.csv};
				\addlegendentry{$\sigma_\eta=1$}
				\addplot[color=black, mark=none, domain=0.5:2e-4]{1};
				\legend{}
			\end{axis}
		\end{tikzpicture}
		\caption{Expectation of the approximate likelihood ratio. \label{fig:unlikelytolikelyratio}}
	\end{subfigure}
	\par\bigskip
	\begin{subfigure}{\linewidth}
		\centering
		\begin{tikzpicture}[trim axis left]
			\begin{axis}[
				xlabel={Maximum solver standard deviation ($\sigma_\delta$)},
				ylabel={$\mathbb{E}\left[\min \left(\frac{\hat{l}(D_1)}{\hat{l}(D_2)},1\right) \right]$},
				xmode=log,
				ymax=1.0,
				ymin=-0.0,
				xmin=1e-3,
				xmax=0.3,
				width=0.9\textwidth,
				height=0.5\textwidth,
				legend pos=south west,
				legend columns=2
				]
				\addplot[
				color=OliveGreen,
				mark=pentagon*,
				thick
				]
				table[
				col sep=comma,
				x=maxsolverstdsigma0.01,
				y=meansigma0.01,
				] {numerical_results/plotdata_acceptance_rate_true_0.1_current_0.08_proposal_0.1_samples_1000.csv};
				\addlegendentry{$\sigma_\eta=0.01$}
				\addplot[
				color=Fuchsia,
				mark=square*,
				thick
				]
				table[
				col sep=comma,
				x=maxsolverstdsigma0.025,
				y=meansigma0.025,
				] {numerical_results/plotdata_acceptance_rate_true_0.1_current_0.08_proposal_0.1_samples_1000.csv};
				\addlegendentry{$\sigma_\eta=0.025$}
				\addplot[
				color=Turquoise,
				mark=asterisk,
				thick
				]
				table[
				col sep=comma,
				x=maxsolverstdsigma0.05,
				y=meansigma0.05,
				] {numerical_results/plotdata_acceptance_rate_true_0.1_current_0.08_proposal_0.1_samples_1000.csv};
				\addlegendentry{$\sigma_\eta=0.05$}
				\addplot[
				color=Mahogany,
				mark=diamond*,
				thick
				]
				table[
				col sep=comma,
				x=maxsolverstdsigma0.1,
				y=meansigma0.1,
				] {numerical_results/plotdata_acceptance_rate_true_0.1_current_0.08_proposal_0.1_samples_1000.csv};	
				\addlegendentry{$\sigma_\eta=0.1$}
				\addplot[
				color=Dandelion,
				mark=pentagon*,
				thick
				]
				table[
				col sep=comma,
				x=maxsolverstdsigma0.25,
				y=meansigma0.25,
				] {numerical_results/plotdata_acceptance_rate_true_0.1_current_0.08_proposal_0.1_samples_1000.csv};
				\addlegendentry{$\sigma_\eta=0.25$}
				\addplot[
				color=WildStrawberry,
				mark=heart,
				thick
				]
				table[
				col sep=comma,
				x=maxsolverstdsigma0.5,
				y=meansigma0.5,
				] {numerical_results/plotdata_acceptance_rate_true_0.1_current_0.08_proposal_0.1_samples_1000.csv};
				\addlegendentry{$\sigma_\eta=0.5$}
				\addplot[
				color=MidnightBlue,
				mark=*,
				thick
				]
				table[
				col sep=comma,
				x=maxsolverstdsigma1,
				y=meansigma1,
				] {numerical_results/plotdata_acceptance_rate_true_0.1_current_0.08_proposal_0.1_samples_1000.csv};
				\addlegendentry{$\sigma_\eta=1$}
				\addplot[color=black, mark=none, domain=0.5:2e-4]{0};
				\addplot[color=black, mark=none, domain=0.5:2e-4]{1};
			\end{axis}
		\end{tikzpicture}
		\caption{Expectation of the approximate likelihood ratio, truncated to the range $[0,1]$. \label{fig:unlikelytolikelyacceptance}}
	\end{subfigure}
	\caption{Computing the expectation of the approximate likelihood ratio given a discrete solution of the problem \eqref{eq:diff_eq}--\eqref{eq:diff_eq_boundary} perturbed by synthetic observation noise for $D^\ast = D_1 =0.1$ and $D_2=0.08$. \label{fig:unlikelytolikely}}
\end{figure}
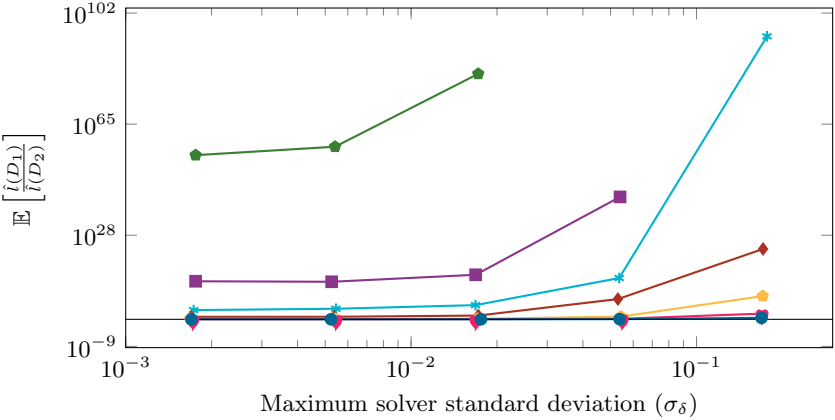
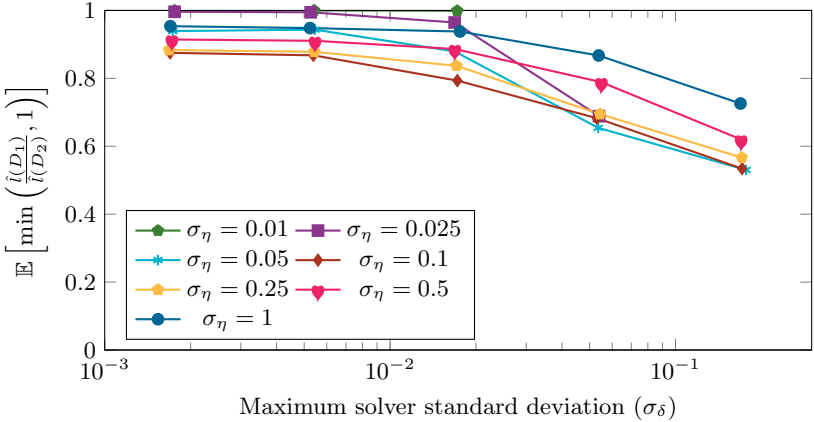

Next, in Figure~\ref{fig:likelytolikely}, we consider the case where $D_2 = D_1 = D^\ast = 0.1$, i.e., both the current state and proposal match the true parameter value. We plot the expectation of the approximate likelihood ratio in Figure~\ref{fig:likelytolikelyratio}. Here we observe a significant blowup in the expectation at approximately $\sigma_\eta=\sigma_\delta$, as predicted by Theorem~\ref{thm:moments}. For sufficiently small values of $\sigma_\delta$, we observe that the expectation of the likelihood ratio converges to the true value of 1. To better understand the divergence of the approximate likelihood ratio, we plot the expectation of the minimum of the ratio and 1 in Figure~\ref{fig:likelytolikelyacceptance}, which in the case of $D_1=D_2$ directly corresponds to the acceptance probability in the Metropolis-Hastings algorithm, see e.g.~\cite[Ch. 7]{Robert2004}. Here, it becomes clear that when the approximate likelihood ratio diverges, the acceptance probability converges to the value 0.5. This behavior gives insight into the precise mechanism behind the divergence. Namely, in this regime, the stochastic error in the approximate likelihood evaluations is sufficiently large that the ratio consists of a numerator and denominator with wildly differing orders of magnitude. Hence, the random variable converges to a distribution with two states, one of which diverges to infinity, the other converges to zero. Each state has probability 0.5.

We now present similar plots for transitions between the values 0.08 and 0.1. In Figure~\ref{fig:likelytounlikely}, we consider the transition from $D_2=D^\ast=0.1$ to $D_1=0.08$. Here we note that Figure~\ref{fig:likelytounlikelyacceptance} now does not in general represent the acceptance probability as we neglect the transition probabilities of the proposal distribution in the Metropolis-Hastings algorithm and the prior distribution. However, we still gain insight into the behavior of the acceptance probability. We also note that this subfigure will correspond to the acceptance probability when considering uninformed prior distributions and symmetric proposal distributions.

We observe a similar blow-up pattern in Figure~\ref{fig:likelytounlikelyratio} to that in Figure~\ref{fig:likelytolikelyratio}, but now observe that not all curves converge to 1 as $\sigma_\delta$ decreases. For sufficiently small values of $\sigma_\eta$, the computed approximate likelihood ratio correctly takes small values to denote the low likelihood of the transition being accepted. In Figure~\ref{fig:likelytounlikelyacceptance}, we see a similar convergence to the value 0.5 as $\sigma_\delta$ increases. As opposed to Figure~\ref{fig:likelytolikelyacceptance}, we note that the curves converge to different values for small $\sigma_\delta$, indicating that we are able to retrieve a sensible acceptance probability.

Similarly, Figure~\ref{fig:unlikelytolikely} presents equivalent plots for the transition from 0.1 to 0.08. In Figure~\ref{fig:unlikelytolikelyratio}, we see a similar behavior to the previous figures, in that the curves start to blow up around $\sigma_\eta=\sigma_\delta$. For small values of $\sigma_\delta$, we see the inverse of the behavior observed in Figure~\ref{fig:likelytounlikelyratio}. Namely, the expectation of the ratio converges to one from above as $\sigma_\eta$ increases. Considering Figure~\ref{fig:unlikelytolikelyacceptance}, we note a similar convergence towards the value 0.5 for larger values of $\sigma_\delta$ as in the previous figures. However, for smaller values of $\sigma_\delta$ we note that the curves are not sequentially ordered in their values of $\sigma_\eta$. The leftmost part of the curves decrease as $\sigma_\eta$ increases from 0.01 through 0.1, but then increase again as $\sigma_\eta$ further increases to 1. It appears that the initial decrease is due to the decrease of the corresponding sequence of curves in Figure~\ref{fig:unlikelytolikelyratio}, i.e., when the expectation of the approximate likelihood ratio decreases, the minimum of 1 and this ratio will more often take values smaller than 1. As $\sigma_\eta$ increases further, it is likely that the likelihood function becomes sufficiently flat that the variance induced by the Monte Carlo solver no longer has a strong impact on the computed approximate likelihood value. Hence, the variance of the approximate likelihood ratio will decrease, and the truncated ratio converges in expectation to 1.

\section{Conclusions}
\label{sec:conclusion}

We have conducted an initial study on the effect of using Monte Carlo simulations to evaluate approximate likelihood ratios, as used in Metropolis-Hastings. In Section~\ref{sec:MCMC} we presented a theorem, stating that the expectation of this ratio exists if the difference between the covariance matrix of the Monte Carlo error in the denominator $\delta_2$ and that of the observation error $\eta$ is positive definite. As the variance of the approximation error further decreases, the distribution of this ratio gains an increasing number of moments. We conducted a numerical study of the behavior of this ratio, confirming the existence criterion outlined in the theorem for transitions between two states with differing true likelihood values. When the forward model errors variance grows too large, we observe that the approximate likelihood ratio, truncated to the range $[0,1]$, converges in expectation to the value 0.5. Hence, the approximate likelihood evaluations no longer provide any useful information for the accept-reject step of the Metropolis-Hastings algorithm.

We conclude that it is likely feasible to use Metropolis-Hastings sampling in settings where one must accept relatively large stochastic errors. However, it is important to assure that the forward model variance does not surpass the limits established in Theorem~\ref{thm:moments}. In addition, there is a clear bias in the approximate likelihood ratio and corresponding acceptance probability that decreases with the the forward model variance. These observation are useful for settings such as Monte Carlo simulation, where one can explicitly control the variance of the forward model. Theorem~\ref{thm:moments} also provides closed-form expressions for the moments of the distribution of the approximate likelihood ratio. We expect these expressions to be useful for algorithm design. One can use, e.g., estimates for the variance of the approximate likelihood ratio to produce error bounds for the computed acceptance probability and determine how many additional trajectories must be simulated in the forward model to make a clear acceptance or rejection decision in the Metropolis-Hastings algorithm.

\section*{Acknowledgments}
We thank the anonymous reviewers for their work on reviewing this manuscript. We thank Robert Scheichl and Josef Martínek for multiple productive discussions on this topic. Emil Løvbak was funded by the Deutsche Forschungsgemeinschaft (DFG, German Research Foundation) – Project-ID 563450842. The authors acknowledge support by the state of Baden-Württemberg through bwHPC.

\bibliographystyle{spmpsci}
\bibliography{noisyMCMC,Martinek}

\end{document}